\documentclass[12pt, a4paper]{article}
\usepackage[top=2.54cm, bottom= 2.54cm, left=2.54cm, right=2.54cm]{geometry}
\def\Dj{\hbox{D\kern-.73em\raise.30ex\hbox{-}
\raise-.30ex\hbox{}}}
\def\dj{\hbox{d\kern-.33em\raise.80ex\hbox{-}
\raise-.80ex\hbox{\kern-.40em}}}
\usepackage{pstricks}
\usepackage{graphicx}
\usepackage{amsmath,amsthm,amsfonts,amssymb,amscd,cite}

\allowdisplaybreaks

\newtheorem{thm}{Theorem}

\title{\Large \bf New Eccentricity Based Topological Indices of Total Transformation Graphs}
\author{S. M. Hosamani$^{*}$, S. S. Shirakol$^{\ddag}$, M. V. Kalyanshetti$^{**}$ and I. N. Cangul$^{\dag}$\\[5mm]
{\it \normalsize $^{*}$ Department of Mathematics, Rani Channamma University,}\\
{\it \normalsize Belagavi- 591 156, Karnataka, India}\\ 
\normalsize e-mail: {\tt sunilkumar.rcu@gmail.com}\\[-1mm]
{\it \normalsize $^{\ddag}$Department of Mathematics, SDMCET, Dharwad, Karnataka, India} \\[-1mm]
\normalsize e-mail: {\tt shailajashirkol@gmail.com}\\[-1mm]
{\it \normalsize $^{**}$Department of Mathematics, Jain College of Engineering, Belagavi, Karnataka, India} \\[-1mm]
\normalsize e-mail: {\tt kalyanshettimanjula@gmail.com}\\
{\it \normalsize $^{\dag}$ Department of Mathematics, Bursa Uludag University,}\\
{\it \normalsize Gorukle Bursa, 16059, Turkey}\\[-1mm]
\normalsize e-mail: {\tt cangul@uludag.edu.tr}}
\date{}
\begin{document}

\maketitle

\begin{abstract}
The eccentric-connectivity index of a graph $G$ is the sum of the products of the eccentricity and the degree of each vertex in $G$. In this paper, we define four new invariants related to the eccentric-connectivity index and obtain upper bounds for total transformation graphs which are some generalizations of total graph.
\\[3mm]
{\bf \AmS Subject Classification:} 05C09; 05C92.
\end{abstract}
\baselineskip=0.30in

\section{Introduction}
Throughout this paper, we only consider finite undirected simple graphs having no loops and multiple edges. Let $G=(V,E)$ be a graph with order $|V| = n$ and size $|E| = m$. The degree of a vertex $v \in G$ is denoted by $deg_{G}(v)$ and defined as $deg_{G}(v) = |\{u :\ uv \in E(G)\}|$. The eccentricity of a vertex $v \in G$ is the largest distance between $v$ and $u$ for all $u \in V(G)$. We follow \cite{4} for unexplained terminology and notation.\\

Recently, topological indices are playing vital role in QSPR/QSAR studies due to their predicting power. One of the oldest topological indices is the first Zagreb index \cite{4} which has been studied extensively by many researchers \cite{1,3,6}. It is defined by
\begin{eqnarray}
  M_{1}(G) &=& \sum\limits_{i=1}^{n} deg(v_{i})^{2}.
\end{eqnarray}
In fact, we can re-write (1) as
\begin{eqnarray}
  M_{1}(G) &=& \sum\limits_{uv\in E(G)} deg(u)+deg(v).
\end{eqnarray}

The forgotten index \cite{2} is defined by 
\begin{eqnarray}
F(G) &=& \sum_{u \in V(G)}deg(u)^{3}.
 \end{eqnarray}

Sharma et.al. \cite{7} have put forward a novel topological index called the eccentric-connectivity index $ECI(G)$ for a molecular graph $G$. This index is defined as follows:
\begin{eqnarray}
  ECI(G) &=& \sum\limits_{i=1}^{n} e(v_{i})deg(v_{i}).
\end{eqnarray}
\vspace{2mm}

Motivated by the eccentric-connectivity index, here we introduce the inverse eccentricity connectivity index $I_{ECI}$, the first Zagreb eccentricity connectivity index $M_{ECI}^{1}$, the first eccentricity connectivity index $ECI^{1}$ and the first Zagreb eccentricity connectivity index $M_{ECI^{1}}^{1}$ as follows:

\begin{eqnarray}
 I_{ECI}  &=&  \sum\limits_{i=1}^{n} (e(v_{i})deg(v_{i}))^{-1},
\end{eqnarray}

\begin{eqnarray}
  M_{ECI}^{1}(G) &=& \sum\limits_{i=1}^{n} (e(v_{i})deg(v_{i})^{2}),
\end{eqnarray}

\begin{eqnarray}
  ECI^{1}(G) &=& \sum\limits_{i=1}^{n} e(v_{i})^{2}deg(v_{i}),
\end{eqnarray}

\begin{eqnarray}
  M_{ECI^{1}}^{1}(G) &=& \sum\limits_{i=1}^{n} e(v_{i})^{2}deg(v_{i})^{2}.
\end{eqnarray}

Let $G$ be a graph. The total graph denoted by $T(G)$ of $G$ has $V(G)  \cup E(G)$ as its vertex set where two vertices of $T(G)$ are adjacent if and only if they are adjacent or incident in $G$. Inspired by the total graph, Wu and Meng \cite{8} have generalized the total graph by defining the following transformation graphs:
\vspace{2mm}

Let $G=(V,E)$ be a graph and $x,y,z$ be three variables taking values $+$ or $-$. The total transformation graph $G^{xyz}$ is a graph  having
$V(G) \cup E(G)$ as the vertex set so that for $\alpha, \beta \in V(G) \cup E(G)$, $\alpha$ and $\beta$ are adjacent in $G^{xyz}$ if and only if either
\begin{enumerate}
\item $\alpha, \beta \in V(G)$, $\alpha$, $\beta$ are adjacent in $G$ if $x = +$ and $\alpha$ and $\beta$ are not adjacent in $G$ if $x = -$;\\
or
\item $\alpha, \beta \in E(G)$, $\alpha$, $\beta$ are adjacent in $G$ if $y = +$ and $\alpha$ and $\beta$ are not adjacent in $G$ if $y = -$;\\
or
\item $\alpha \in V(G)$ and $\beta \in E(G)$, $\alpha$, $\beta$ are incident in $G$ if $z = +$ and $\alpha$ and $\beta$ are not incident in $G$ if $z = -$.
\end{enumerate}

\noindent\textbf{Note 1.} Since there are eight distinct  3-permutations of $\{+, -\}$, we obtain eight  graphical transformations of $G$. It is
interesting to see that $G^{+++}$ is exactly  the total graph $T(G)$ of $G$  and $G^{---}$  is the complement of $T(G)$. Also for a  given graph $G$, $G^{++-}$ and $G^{--+}$, $G^{+-+}$ and $G^{-+-}$,  $G^{-++}$  and $G^{+--}$ are the other three pairs of  complementary graphs. For basic properties of these transformation graphs, we can refer to \cite{5,9,10}.

\section{Results}
In this section, we obtain some new upper bounds for the eccentric-connectivity index of total transformation graphs in terms of order, size and the first Zagreb index.

\begin{thm}
Let $G$ be an $(n,m)$ graph. Then
\begin{eqnarray*}
  I_{ECI}(G^{+++}) &\leq& \frac{1}{2}I_{ECI}(G)+ \frac{1}{4}m + \xi^{-1}(G) + M_{1}^{-1}(G); \\
  M_{ECI}^{1}(G^{+++}) &\leq & 4M^{1}_{ECI}(G) + F(G) + 2M_{2}(G) + 4M_{1}(G)+ F_{ECI}(G) +2M^{2}_{ECI}(G);\\
  ECI^{1}(G^{+++})&\leq& 2ECI^{1}(G) + 4ECI(G) + M_{1}(G) + 4m + ECI^{1}(G)(G) + 2 M^{1}_{ECI}(G);\\
  M_{ECI^{1}}^{1}(G^{+++})&\leq& 4M^{1}_{ECI^{1}}(G) + 4 M_{1}(G) + 8 M^{1}_{ECI}(G) + F(G) + 2M_{2}(G) + F_{ECI'}(G) \\
   &+& 2M^{2}_{ECI^{1}}(G) + 2 F_{ECI}(G) + 4M^{2}_{ECI}(G).
\end{eqnarray*}
\end{thm}
\begin{proof}
Let $G = (V,E)$ be a graph with $V(G) = \{v_{1}, v_{2}, v_{3}, \cdots, v_{n}\}$ and the edge set $E(G) = \{e_{1}, e_{2}, e_{3}, \cdots, e_{m}\}$. Then
$V(G^{+++}) = \{v_{1}, v_{2}, v_{3}, \cdots, v_{n}, e_{1}, e_{2}, e_{3}, \cdots, e_{m}\}$. Clearly $|V(G^{+++})|=m+n$. Further, $diam(G) \leq diam(G^{+++}) \leq diam(G) +1$. Since for every $v \in V(G)$, $e_{G}(v) \leq diam(G)$, we get $e_{_{G^{+++}}}(u) \leq e_{_{G}}(u) +1$ for every $u \in G^{+++}$. Let $u_{i}\in V(G^{+++})$ be the corresponding vertex to $v_{i} \in V(G)$ and $u_{j} \in V(G^{+++})$ be the corresponding vertex to $e_{j} \in E(G)$ in $G^{+++}$. Then $deg_{_{G^{+++}}}(u_{i}) = 2 deg_{_{G}}(v_{i})$ and $deg_{_{G^{+++}}}(u_{j}) =deg_{_{G}}(v_{i})+ deg_{_{G}}(v_{j})$ where $e_{j} = v_{i}v_{j}$. Therefore we have to take care of the following cases:
\begin{description}
  \item[Case 1.]
  \begin{eqnarray*}
  I_{ECI}(G^{+++}) &=& \sum\limits_{i=1}^{n} \frac{1}{e_{_{G^{+++}}}(u)\cdot deg_{_{G^{+++}}}(u)} \\
  &=& \sum\limits_{u_{i} \in V(G^{+++}) \cap V(G)} \frac{1}{e_{_{G^{+++}}}(u_{i})\cdot deg_{_{G^{+++}}}(u_{i})}\\
   &+& \sum\limits_{u_{j} \in V(G^{+++}) \cap E(G)} \frac{1}{e_{_{G^{+++}}}(u_{j})\cdot deg_{_{G^{+++}}}(u_{j})}.
\end{eqnarray*}
Since $e_{_{G^{+++}}}(u) \leq e_{_{G}}(u) +1$, we get
\begin{eqnarray*}
  ECI(G^{+++}) &\leq&\sum\limits_{u_{i} \in V(G)} \frac{1}{\big[(e_{_{G}}(u_{i}) +1)\cdot 2deg_{_{G}}(u_{i}) \big]} \\
  &+& \sum\limits_{u_{j}, u_{k} \in E(G)}\frac{1}{\big[ (e_{_{G}}(u_{j}) +1)\cdot (deg_{_{G}}(u_{j}) +deg_{_{G}}(u_{k})) \big]} \\
  &=& \frac{1}{2} \sum\limits_{u_{i} \in V(G)} \frac{1}{e_{_{G}}(u_{i})deg_{_{G}}(u_{i})} + \frac{1}{2} \sum\limits_{u_{i} \in V(G)} \frac{1}{deg_{_{G}}(u_{i})}\\
   &+& \sum\limits_{u_{j}u_{k}\in E(G)} \frac{1}{e_{_{G}}(u_{j})(deg_{_{G}}(u_{j}) +deg_{_{G}}(u_{k}))} + \sum\limits_{u_{j}u_{k} \in E(G)} \frac{1}{(deg_{_{G}}(u_{j}) +deg_{_{G}}(u_{k}))}\\
   &\leq& \frac{1}{2}I_{ECI}(G)+ \frac{1}{2}ID(G) + \xi^{-1}(G) + M_{1}^{-1}(G)
\end{eqnarray*}
where $$\xi^{-1}(G) = \sum\limits_{u_{j}u_{k}\in E(G)} \frac{1}{e_{_{G}}(u_{j})(deg_{_{G}}(u_{j}) +deg_{_{G}}(u_{k}))}$$ and $$M_{1}^{-1}(G) =\sum\limits_{u_{j}u_{k} \in E(G)} \frac{1}{(deg_{_{G}}(u_{j}) +deg_{_{G}}(u_{k}))}.$$
\item[Case 2.]
 \begin{eqnarray*}
  M_{ECI}^{1}(G) &=& \sum\limits_{i=1}^{n} e(v_{i})deg(v_{i})^{2}\\
  &=& \sum\limits_{u_{i} \in V(G^{+++}) \cap V(G)} e_{_{G^{+++}}}(u_{i})\cdot deg_{_{G^{+++}}}(u_{i})^{2}\\
   &+& \sum\limits_{u_{j} \in V(G^{+++}) \cap E(G)} e_{_{G^{+++}}}(u_{j})\cdot deg_{_{G^{+++}}}(u_{j})^{2}.
\end{eqnarray*}
Since $e_{_{G^{+++}}}(u) \leq e_{_{G}}(u) +1$, we obtain
\begin{eqnarray*}
  M^{1}_{ECI}(G^{+++}) &\leq& \sum\limits_{u_{i} \in V(G)} \big[(e_{_{G}}(u_{i}) +1)\cdot (2deg_{_{G}}(u_{i}))^{2} \big] \\
  &+& \sum\limits_{u_{j}, u_{k} \in E(G)}\big[ (e_{_{G}}(u_{j}) +1)\cdot (deg_{_{G}}(u_{j}) +deg_{_{G}}(u_{k}))^{2} \big] \\
  &=& 4 \sum\limits_{u_{i} \in V(G)}\big[e_{_{G}}(u_{i})deg_{_{G}}(u_{i}))^{2} \big] + 4\sum\limits_{u_{i} \in V(G)} deg_{_{G}}(u_{i}))^{2} \\
  &+& \sum\limits_{u_{j}, u_{k} \in E(G)}e_{_{G}}(u_{j})(deg_{_{G}}(u_{j}) +deg_{_{G}}(u_{k}))^{2} + \sum\limits_{u_{j}, u_{k} \in E(G)}(deg_{_{G}}(u_{j}) +deg_{_{G}}(u_{k}))^{2} \\
  &=& 4M^{1}_{ECI}(G) + 4M_{1} + \sum\limits_{u_{j}, u_{k} \in E(G)}e_{_{G}}(u_{j})(deg_{_{G}}(u_{j})^{2} +deg_{_{G}}(u_{k})^{2}) \\
  &+& 2 \sum\limits_{u_{j}, u_{k} \in E(G)}e_{_{G}}(u_{j})(deg_{_{G}}(u_{j})deg_{_{G}}(u_{k})) + \sum\limits_{u_{j}, u_{k} \in E(G)}(deg_{_{G}}(u_{j})^{2} +deg_{_{G}}(u_{k})^{2}) \\
  &+& 2 \sum\limits_{u_{j}, u_{k} \in E(G)}deg_{_{G}}(u_{j})deg_{_{G}}(u_{k})\\
  &=& 4M^{1}_{ECI}(G) + F(G) + 2M_{2}(G) + 4M_{1}(G)+ F_{ECI}(G) +2M^{2}_{ECI}(G)
\end{eqnarray*}
where $$F_{ECI}(G) = \sum\limits_{u_{j}, u_{k} \in E(G)}e_{_{G}}(u_{j})(deg_{_{G}}(u_{j})^{2} +deg_{_{G}}(u_{k})^{2})$$ and
$$M^{2}_{ECI}(G)= \sum\limits_{u_{j}, u_{k} \in E(G)}e_{_{G}}(u_{j})deg_{_{G}}(u_{j})deg_{_{G}}(u_{k}).$$

\item[Case 3.]
  \begin{eqnarray*}
  ECI^{1}(G^{+++}) &=& \sum\limits_{i=1}^{n} e(v_{i})^{2}deg(v_{i})\\
  &=& \sum\limits_{u_{i} \in V(G^{+++}) \cap V(G)} e_{_{G^{+++}}}(u_{i})^{2}\cdot deg_{_{G^{+++}}}(u_{i})\\
   &+& \sum\limits_{u_{j} \in V(G^{+++}) \cap E(G)} e_{_{G^{+++}}}(u_{j})^{2}\cdot deg_{_{G^{+++}}}(u_{j}).
\end{eqnarray*}
Since $e_{_{G^{+++}}}(u) \leq e_{_{G}}(u) +1$, we reduce that
\begin{eqnarray*}
  ECI^{1}(G^{+++}) &\leq& \sum\limits_{u_{i} \in V(G)} \big[(e_{_{G}}(u_{i}) +1)^{2}\cdot (2deg_{_{G}}(u_{i})) \big] \\
  &+& \sum\limits_{u_{j}, u_{k} \in E(G)}\big[ (e_{_{G}}(u_{j}) +1)^{2}\cdot (deg_{_{G}}(u_{j}) +deg_{_{G}}(u_{k})) \big] \\
  &=& 2 \sum\limits_{u_{i} \in V(G)} \big[(e_{_{G}}(u_{i})^{2} +1+ 2(e_{_{G}}(u_{i})))(deg_{_{G}}(u_{i})) \big] \\
  &+& \sum\limits_{u_{j}, u_{k} \in E(G)}\big[(e_{_{G}}(u_{i})^{2} +1+ 2(e_{_{G}}(u_{i})))(deg_{_{G}}(u_{j}) +deg_{_{G}}(u_{k})) \big] \\
  &=& 2 \sum\limits_{u_{i} \in V(G)} \big[e_{_{G}}(u_{i})^{2}(deg_{_{G}}(u_{i}))\big] + 2 \sum\limits_{u_{i} \in V(G)} deg_{_{G}}(u_{i})\\
   &+& 4\sum\limits_{u_{i} \in V(G)} e_{_{G}}(u_{i})(deg_{_{G}}(u_{i})) \\
  &+& \sum\limits_{u_{j}, u_{k} \in E(G)}\big[ e_{_{G}}(u_{i})^{2}(deg_{_{G}}(u_{j}) +deg_{_{G}}(u_{k}))\big] + \sum\limits_{u_{j}, u_{k} \in E(G)} \big[ (deg_{_{G}}(u_{j}) +deg_{_{G}}(u_{k}))\big] \\
  &+& 2\sum\limits_{u_{j}, u_{k} \in E(G)} e_{_{G}}(u_{i}) (deg_{_{G}}(u_{j}) +deg_{_{G}}(u_{k})) \\
  &=& 2ECI^{1}(G) + 4ECI(G) + M_{1}(G) + 4m + ECI^{1}(G)(G) + 2 M^{1}_{ECI}(G).
\end{eqnarray*}
\item[Case 4.]
\begin{eqnarray*}
  M^{1}_{ECI^{1}}(G^{+++}) &=& \sum\limits_{i=1}^{n} e(v_{i})^{2}deg(v_{i})^{2}\\
  &=& \sum\limits_{u_{i} \in V(G^{+++}) \cap V(G)} e_{_{G^{+++}}}(u_{i})^{2}\cdot deg_{_{G^{+++}}}(u_{i})^{2}\\
   &+& \sum\limits_{u_{j} \in V(G^{+++}) \cap E(G)} e_{_{G^{+++}}}(u_{j})^{2}\cdot deg_{_{G^{+++}}}(u_{j})^{2}.
\end{eqnarray*}
As $e_{_{G^{+++}}}(u) \leq e_{_{G}}(u) +1$, we get
\begin{eqnarray*}
  ECI^{1}(G^{+++}) &\leq& \sum\limits_{u_{i} \in V(G)} (e_{_{G}}(u_{i}) +1)^{2}\cdot (2deg_{_{G}}(u_{i}))^{2}\\
  &+& \sum\limits_{u_{j}, u_{k} \in E(G)} (e_{_{G}}(u_{j}) +1)^{2}\cdot (deg_{_{G}}(u_{j}) +deg_{_{G}}(u_{k}))^{2}\\
  &=& 4 \sum\limits_{u_{i} \in V(G)} (e_{_{G}}(u_{i})^{2} +1+ 2e_{_{G}}(u_{i}))deg_{_{G}}(u_{i})^{2} \\
  &+& \sum\limits_{u_{j}, u_{k} \in E(G)}(e_{_{G}}(u_{i})^{2} +1+ 2(e_{_{G}}(u_{i})))(deg_{_{G}}(u_{j}) +deg_{_{G}}(u_{k}))^{2} \\
  &=& 4 \sum\limits_{u_{i} \in V(G)} e_{_{G}}(u_{i})^{2}(deg_{_{G}}(u_{i}))^{2} + 4 \sum\limits_{u_{i} \in V(G)} deg_{_{G}}(u_{i})^{2}\\
   &+& 8\sum\limits_{u_{i} \in V(G)} e_{_{G}}(u_{i})deg_{_{G}}(u_{i})^{2}\\
  &+& \sum\limits_{u_{j}, u_{k} \in E(G)} e_{_{G}}(u_{i})^{2}(deg_{_{G}}(u_{j}) +deg_{_{G}}(u_{k}))^{2}\\
  &+& \sum\limits_{u_{j}, u_{k} \in E(G)} (deg_{_{G}}(u_{j}) +deg_{_{G}}(u_{k}))^{2}\\
  &+& 2\sum\limits_{u_{j}, u_{k} \in E(G)} e_{_{G}}(u_{i}) (deg_{_{G}}(u_{j}) +deg_{_{G}}(u_{k}))^{2}\\
  &=& 4M^{1}_{ECI^{1}}(G) + 4 M_{1}(G) + 8 M^{1}_{ECI}(G) + F(G) + 2M_{2}(G) + F_{ECI^{1}}(G)\\
   &+& 2M^{2}_{ECI^{1}}(G) + 2 F_{ECI}(G) + 4M^{2}_{ECI}(G)
\end{eqnarray*}
\end{description}
where $$M^{2}_{ECI^{1}}(G) =  \sum\limits_{u_{j}, u_{k} \in E(G)}e_{_{G}}(u_{j})^{2}deg_{_{G}}(u_{j})deg_{_{G}}(u_{k})$$ and \\
   $$F_{ECI^{1}}(G) = \sum\limits_{u_{j}, u_{k} \in E(G)}e_{_{G}}(u_{j})^{2}(deg_{_{G}}(u_{j})^{2} +deg_{_{G}}(u_{k})^{2}).$$
\end{proof}

\begin{thm}
Let $G$ be an $(n,m)$ graph. Then
\begin{eqnarray*}
  I_{ECI}(G^{---}) &\leq& \frac{1}{3(m+n-1)n} + \frac{1}{3(m+n-1)m}-\frac{1}{12m}-\frac{1}{3}M_{1}^{-1}(G),\\
  M_{ECI}^{1}(G^{---}) &\leq & 3F(G) +6M_{2}(G) -6(m++n-3)M_{1}(G)+3(m+n)(m+n-1)^{2},\\
  &-&24m(m+n-1)\\
  ECI^{1}(G^{---})&\leq& 9n(m+n-1) +9m(m+n) -36m -9M_{1}(G),\\
  M_{ECI^{1}}^{1}(G^{---})&\leq& 9F(G) + 18M_{2}(G) -18(m+n-3)M_{1}(G) \\
  &+&9(m+n)(m+n-1)^{2} -36m(m+n-1). 
\end{eqnarray*}
\end{thm}
\begin{proof}
Let $G = (V,E)$ be a graph with $V(G) = \{v_{1}, v_{2}, v_{3}, \cdots, v_{n}\}$ and the edge set $E(G) = \{e_{1}, e_{2}, e_{3}, \cdots, e_{m}\}$. Then
$V(G^{---}) = \{v_{1}, v_{2}, v_{3}, \cdots, v_{n}, e_{1}, e_{2}, e_{3}, \cdots, e_{m}\}$. Clearly $|V(G^{---})|=m+n$. Further, $diam(G^{---}) \leq 3$. As for every $v \in V(G)$, $e_{G}(v) \leq diam(G)$, we have $e_{_{G^{---}}}(u) \leq 3$ for every $u \in G^{---}$. Let $u_{i}\in V(G^{---})$ be the corresponding vertex to $v_{i} \in V(G)$ and $u_{j} \in V(G^{---})$ be the corresponding vertex to $e_{j} \in E(G)$ in $G^{---}$. Then $deg_{_{G^{---}}}(u_{i}) = m+n-1-2 deg_{_{G}}(v_{i})$ and $deg_{_{G^{---}}}(u_{j}) =m+n-1-(deg_{_{G}}(v_{i})+ deg_{_{G}}(v_{j}))$ where $e_{j} = v_{i}v_{j}$. Therefore
\begin{description}
  \item[Case 1.]
  \begin{eqnarray*}
  I_{ECI}(G^{---}) &=& \sum\limits_{i=1}^{n} \frac{1}{e_{_{G^{---}}}(u)\cdot deg_{_{G^{---}}}(u)}\\
  &=& \sum\limits_{u_{i} \in V(G^{---}) \cap V(G)} \frac{1}{e_{_{G^{---}}}(u_{i})\cdot deg_{_{G^{---}}}(u_{i})} \\
  &+& \sum\limits_{u_{j} \in V(G^{---}) \cap E(G)} \frac{1}{e_{_{G^{---}}}(u_{j})\cdot deg_{_{G^{---}}}(u_{j})}
\end{eqnarray*}
Since $e_{_{G^{---}}}(u) \leq 3$, we deduce that
\begin{eqnarray*}
  I_{ECI}(G^{---}) &\leq&\sum\limits_{u_{i} \in V(G)} \frac{1}{\big[3\cdot (m+n-1-2(deg_{_{G}}(u_{i}))) \big]}\\
   &+& \sum\limits_{u_{j}, u_{k} \in E(G)}\frac{1}{\big[ 3\cdot (m+n-1-(deg_{_{G}}(u_{j}) +deg_{_{G}}(u_{k}))) \big] }\\
  &=& \frac{1}{3(m+n-1)n} - \frac{1}{6\sum\limits_{u_{i} \in V(G)}(deg_{_{G}}(u_{i})) }\\
  &+& \frac{1}{3(m+n-1)m} - \frac{1}{3\sum\limits_{u_{j}, u_{k} \in E(G)}(deg_{_{G}}(u_{j}) +deg_{_{G}}(u_{k}))} \\
  &\leq&\frac{1}{3(m+n-1)n} + \frac{1}{3(m+n-1)m}-\frac{1}{12m}-\frac{1}{3}M_{1}^{-1}(G).
\end{eqnarray*}

  \item[Case 2.] Similarly, we have
  \begin{eqnarray*}
  M_{ECI}^{1}(G) &=& \sum\limits_{i=1}^{n} (e(v_{i})deg(v_{i})^{2})\\
  &=& \sum\limits_{u_{i} \in V(G^{---}) \cap V(G)} e_{_{G^{---}}}(u_{i})\cdot deg_{_{G^{---}}}(u_{i})^{2}\\
   &+& \sum\limits_{u_{j} \in V(G^{---}) \cap E(G)} e_{_{G^{---}}}(u_{j})\cdot deg_{_{G^{---}}}(u_{j})^{2}.
\end{eqnarray*}
Since $e_{_{G^{---}}}(u) \leq 3$, we obtain
\begin{eqnarray*}
  M^{1}_{ECI}(G^{---}) &\leq& \sum\limits_{u_{i} \in V(G)} \big[3(m+n-1-2deg_{_{G}}(u_{i}))^{2} \big] \\
  &+& \sum\limits_{u_{j}, u_{k} \in E(G)}\big[ 3 (m+n-1-(deg_{_{G}}(u_{j}) +deg_{_{G}}(u_{k})))^{2} \big] \\
  &=& \sum\limits_{u_{i} \in V(G)} 3(m+n-1)^{2} + 12 \sum\limits_{u_{i} \in V(G)} deg_{_{G}}(u_{j}^{2}) -12(m+n-1) \sum\limits_{u_{i} \in V(G)} deg_{_{G}}(u_{j}) \\
  &+& \sum\limits_{u_{j}, u_{k} \in E(G)} 3(m+n-1)^{2} + 3F(G) +6M_{2}(G) -6(m+n-1)M_{1}(G)\\
  &=& 3F(G) +6M_{2}(G) -6(m+n-3)M_{1}(G)+3(m+n)(m+n-1)^{2} \\
  &-&24m(m+n-1).
\end{eqnarray*}

  \item[Case 3.]
  \begin{eqnarray*}
  ECI^{1}(G^{---}) &=& \sum\limits_{i=1}^{n} (e(v_{i})^{2}deg(v_{i}))\\
  &=& \sum\limits_{u_{i} \in V(G^{---}) \cap V(G)} e_{_{G^{---}}}(u_{i})^{2}\cdot deg_{_{G^{---}}}(u_{i})\\
   &+& \sum\limits_{u_{j} \in V(G^{---}) \cap E(G)} e_{_{G^{---}}}(u_{j})^{2}\cdot deg_{_{G^{---}}}(u_{j}).
\end{eqnarray*}
Since $e_{_{G^{---}}}(u) \leq 3$, we obtain
\begin{eqnarray*}
  ECI^{1}(G^{---}) &\leq& \sum\limits_{u_{i} \in V(G)} \big[3^{2}(m+n-1-2deg_{_{G}}(u_{i})) \big] \\
  &+& \sum\limits_{u_{j}, u_{k} \in E(G)}\big[ 3^{2} (m+n-(deg_{_{G}}(u_{j}) +deg_{_{G}}(u_{k}))) \big] \\
  &=& 9n(m+n-1) +9m(m+n) -36m -9M_{1}(G).
\end{eqnarray*}

  \item[Case 4.]
  \begin{eqnarray*}
  M^{1}_{ECI^{1}}(G^{---}) &=& \sum\limits_{i=1}^{n} (e(v_{i})^{2}deg(v_{i})^{2})\\
  &=& \sum\limits_{u_{i} \in V(G^{---}) \cap V(G)} e_{_{G^{---}}}(u_{i})^{2}\cdot deg_{_{G^{---}}}(u_{i})^{2}\\
   &+& \sum\limits_{u_{j} \in V(G^{---}) \cap E(G)} e_{_{G^{---}}}(u_{j})^{2}\cdot deg_{_{G^{---}}}(u_{j})^{2}.
\end{eqnarray*}
Since $e_{_{G^{---}}}(u) \leq 3$, we finally get
\begin{eqnarray*}
  ECI^{1}(G^{---}) &\leq& \sum\limits_{u_{i} \in V(G)} \big[3^{2}(m+n-1-2deg_{_{G}}(u_{i}))^{2} \big] \\
  &+& \sum\limits_{u_{j}, u_{k} \in E(G)}\big[ 3^{2}(m+n-1-(deg_{_{G}}(u_{j}) +deg_{_{G}}(u_{k})))^{2} \big] \\
  &=& 9F(G) + 18M_{2}(G) -18(m+n-3)M_{1}(G) \\
  &+&9(m+n)(m+n-1)^{2} -36m(m+n-1)
\end{eqnarray*}
\end{description}
as asserted.
\end{proof}

\begin{thm}
Let $G$ be an $(n,m)$ graph. Then
\begin{eqnarray*}
  I_{ECI}(G^{++-}) &\leq& \frac{1}{4mn} +\frac{1}{4m(n-4)} +\frac{1}{4}M_{1}^{-1}(G),\\
  M_{ECI}^{1}(G^{++-}) &\leq & 4F(G) +8M_{2}(G) + 8(n-4)M_{1}(G) + 4nm^{2} +4m(n-4)^{2},\\
  ECI^{1}(G^{++-})&\leq& 16M_{1}(G) +16m(n-4) + 16mn,\\
  M_{ECI^{1}}^{1}(G^{++-})&\leq& 16F(G) + 32M_{2}(G) + 16 m^{2}n + 16m(n-4).
\end{eqnarray*}
\end{thm}
\begin{proof}
Let $G = (V,E)$ be a graph with $V(G) = \{v_{1}, v_{2}, v_{3}, \cdots, v_{n}\}$ and the edge set $E(G) = \{e_{1}, e_{2}, e_{3}, \cdots, e_{m}\}$. Then $V(G^{++-}) = \{v_{1}, v_{2}, v_{3}, \cdots, v_{n}, e_{1}, e_{2}, e_{3}, \cdots, e_{m}\}$. Clearly $|V(G^{++-})|=m+n$. Further $diam(G^{++-}) \leq 4$. Since $e_{G}(v) \leq diam(G)$ for every $v \in V(G)$ we have $e_{_{G^{++-}}}(u) \leq 4$ for every $u \in G^{++-}$. Let $u_{i} \in V(G^{++-})$ be the corresponding vertex to $v_{i} \in V(G)$ and $u_{j} \in V(G^{++-})$ be the corresponding vertex to $e_{j} \in E(G)$ in $G^{++-}$. Then $deg_{_{G^{++-}}}(u_{i}) = m$ and $deg_{_{G^{++-}}}(u_{j}) =n-4+(deg_{_{G}}(v_{i})+ deg_{_{G}}(v_{j}))$ where $e_{j} = v_{i}v_{j}$. Therefore we discuss the following cases:
\begin{description}
  \item[Case 1.]
  \begin{eqnarray*}
  I_{ECI}(G^{++-}) &=& \sum\limits_{i=1}^{n} \frac{1}{e_{_{G^{++-}}}(u)\cdot deg_{_{G^{++-}}}(u)}\\
  &=& \sum\limits_{u_{i} \in V(G^{++-}) \cap V(G)} \frac{1}{e_{_{G^{++-}}}(u_{i})\cdot deg_{_{G^{++-}}}(u_{i})} \\
  &+& \sum\limits_{u_{j} \in V(G^{++-}) \cap E(G)} \frac{1}{e_{_{G^{++-}}}(u_{j})\cdot deg_{_{G^{++-}}}(u_{j})}.
\end{eqnarray*}
Since $e_{_{G^{++-}}}(u) \leq 4$, we have
\begin{eqnarray*}
I_{ECI}(G^{++-}) &\leq&\sum\limits_{u_{i} \in V(G)} \frac{1}{\big[4m \big]}+ \sum\limits_{u_{j}, u_{k} \in E(G)}\frac{1}{\big[ 4\cdot (n-4+(deg_{_{G}}(u_{j}) +deg_{_{G}}(u_{k}))) \big]} \\
&\leq& \frac{1}{4mn} +\frac{1}{4m(n-4)} +\frac{1}{4}M_{1}^{-1}(G).
\end{eqnarray*}

\item[Case 2.]
\begin{eqnarray*}
M_{ECI}^{1}(G) &=& \sum\limits_{i=1}^{n} (e(v_{i})deg(v_{i})^{2})\\
&=& \sum\limits_{u_{i} \in V(G^{++-}) \cap V(G)} e_{_{G^{++-}}}(u_{i})\cdot deg_{_{G^{++-}}}(u_{i})^{2}\\
&+& \sum\limits_{u_{j} \in V(G^{++-}) \cap E(G)} e_{_{G^{++-}}}(u_{j})\cdot deg_{_{G^{++-}}}(u_{j})^{2}.
\end{eqnarray*}
Since $e_{_{G^{++-}}}(u) \leq 4$, we obtain
\begin{eqnarray*}
  M^{1}_{ECI}(G^{++-}) &\leq& \sum\limits_{u_{i} \in V(G)} \big[4m^{2}] \\
  &+& \sum\limits_{u_{j}, u_{k} \in E(G)}4\big[(n-4)+(deg_{_{G}}(u_{j}) +deg_{_{G}}(u_{k}))^{2} \big] \\
  &=& 4 nm^{2} + 4m(n-4)^{2}+ 4\sum\limits_{u_{i} \in V(G)} (deg_{_{G}}(u_{j}) +deg_{_{G}}(u_{k}))^{2} + 8(n-4)M_{1}(G)\\
  &=& 4F(G) +8M_{2}(G) + 8(n-4)M_{1}(G) + 4nm^{2} +4m(n-4)^{2}. 
\end{eqnarray*}

  \item[Case 3.]
  \begin{eqnarray*}
  ECI^{1}(G) &=& \sum\limits_{i=1}^{n} e(v_{i})^{2}deg(v_{i})\\
  &=& \sum\limits_{u_{i} \in V(G^{++-}) \cap V(G)} e_{_{G^{++-}}}(u_{i})^{2}\cdot deg_{_{G^{++-}}}(u_{i})\\
   &+& \sum\limits_{u_{j} \in V(G^{++-}) \cap E(G)} e_{_{G^{++-}}}(u_{j})^{2}\cdot deg_{_{G^{++-}}}(u_{j}).
\end{eqnarray*}
Since $e_{_{G^{++-}}}(u) \leq 4$, we deduce that
\begin{eqnarray*}
  ECI^{1}(G^{++-}) &\leq& \sum\limits_{u_{i} \in V(G)} 16m\\
  &+& \sum\limits_{u_{j}, u_{k} \in E(G)}\big[ 4^{2}(n-4+deg_{_{G}}(u_{j}) +deg_{_{G}}(u_{k}))) \big] \\
&=& 16M_{1}(G) +16m(n-4) + 16mn.
\end{eqnarray*}
\item[Case 4.]
\begin{eqnarray*}
  M^{1}_{ECI^{1}}(G) &=& \sum\limits_{i=1}^{n} (e(v_{i})^{2}deg(v_{i})^{2})\\
  &=& \sum\limits_{u_{i} \in V(G^{++-}) \cap V(G)} e_{_{G^{++-}}}(u_{i})^{2}\cdot deg_{_{G^{++-}}}(u_{i})^{2}\\
   &+& \sum\limits_{u_{j} \in V(G^{++-}) \cap E(G)} e_{_{G^{++-}}}(u_{j})^{2}\cdot deg_{_{G^{++-}}}(u_{j})^{2}.
\end{eqnarray*}
Since $e_{_{G^{++-}}}(u) \leq 4$, we finally have
\begin{eqnarray*}
  ECI^{1}(G^{++-}) &\leq& \sum\limits_{u_{i} \in V(G)} \big[4^{2}(m)^{2} \big] \\
  &+& \sum\limits_{u_{j}, u_{k} \in E(G)}\big[ 4^{2}(n-4+(deg_{_{G}}(u_{j}) +deg_{_{G}}(u_{k})))^{2} \big] \\
  &=& 16F(G) + 32M_{2}(G) + 16 m^{2}n + 16m(n-4)
\end{eqnarray*}
\end{description}
as asserted.
\end{proof}

\begin{thm}
Let $G$ be an $(n,m)$ graph. Then
\begin{eqnarray*}
  I_{ECI}(G^{--+}) &\leq& \frac{1}{3n(n+1)} -\frac{1}{6m}+ \frac{1}{3m(m+3)} -\frac{1}{3}M_{1}^{-1}(G),\\
  M_{ECI}^{1}(G^{--+}) &\leq & 3F(G) + 6M_{2}(G) - (6m + 15)M_{1}(G) + 3n(n+1)^{2}\\
  &-&12m(n+1)+ 3m(m+3)^{2}, \\
  ECI^{1}(G^{--+})&\leq& 9n(n+1) +9m(m+3)-18m -9M_{1}(G),\\
  M_{ECI^{1}}^{1}(G^{--+})&\leq&  9F(G) + 18M_{2}(G) - (18m+45)M_{1}(G) + 9n(n+1)^{2} \\
  &+&9m(m+3)^{2} -36m(n+1).
\end{eqnarray*}
\end{thm}
\begin{proof}
Let $G = (V,E)$ be a graph with $V(G) = \{v_{1}, v_{2}, v_{3}, \cdots, v_{n}\}$ and the edge set $E(G) = \{e_{1}, e_{2}, e_{3}, \cdots, e_{m}\}$. Then
$V(G^{--+}) = \{v_{1}, v_{2}, v_{3}, \cdots, v_{n}, e_{1}, e_{2}, e_{3}, \cdots, e_{m}\}$. Clearly $|V(G^{--+})|=m+n$ and $diam(G^{--+}) \leq 3$. Since, for every $v \in V(G)$, $e_{G}(v) \leq diam(G)$, we get $e_{_{G^{--+}}}(u) \leq 3$ for every $u \in G^{--+}$. Let $u_{i}\in V(G^{--+})$ be the corresponding vertex to $v_{i} \in V(G)$ and $u_{j} \in V(G^{--+})$ be the corresponding vertex to $e_{j} \in E(G)$ in $G^{--+}$. Then $deg_{_{G^{--+}}}(u_{i}) = n+1- deg_{_{G}}(v_{i})$ and $deg_{_{G^{--+}}}(u_{j}) =m+3-(deg_{_{G}}(v_{i})+ deg_{_{G}}(v_{j}))$ where $e_{j} = v_{i}v_{j}$. Therefore
\begin{description}
  \item[Case 1.]
 \begin{eqnarray*}
  ECI(G^{--+}) &=& \sum\limits_{i=1}^{n} \frac{1}{e_{_{G^{--+}}}(u)\cdot deg_{_{G^{--+}}}(u)}\\
  &=& \sum\limits_{u_{i} \in V(G^{--+}) \cap V(G)} \frac{1}{e_{_{G^{--+}}}(u_{i})\cdot deg_{_{G^{--+}}}(u_{i})}\\
   &+& \sum\limits_{u_{j} \in V(G^{--+}) \cap E(G)} \frac{1}{e_{_{G^{--+}}}(u_{j})\cdot deg_{_{G^{--+}}}(u_{j})}.
\end{eqnarray*}
Since $e_{_{G^{--+}}}(u) \leq 3$, we get
\begin{eqnarray*}
  ECI(G^{--+}) &\leq&\sum\limits_{u_{i} \in V(G)} \frac{1}{\big[3\cdot (n+1-(deg_{_{G}}(u_{i}))) \big]}\\
   &+& \sum\limits_{u_{j}, u_{k} \in E(G)}\frac{1}{\big[ 3\cdot (m+3-(deg_{_{G}}(u_{j}) +deg_{_{G}}(u_{k}))) \big]} \\
   &\leq& \frac{1}{3n(n+1)} -\frac{1}{6m}+ \frac{1}{3m(m+3)} -\frac{1}{3}M_{1}^{-1}(G).
\end{eqnarray*}

  \item[Case 2.]
  \begin{eqnarray*}
  M_{ECI}^{1}(G) &=& \sum\limits_{i=1}^{n} (e(v_{i})deg(v_{i})^{2})\\
  &=& \sum\limits_{u_{i} \in V(G^{--+}) \cap V(G)} e_{_{G^{--+}}}(u_{i})\cdot deg_{_{G^{--+}}}(u_{i})^{2}\\
   &+& \sum\limits_{u_{j} \in V(G^{--+}) \cap E(G)} e_{_{G^{--+}}}(u_{j})\cdot deg_{_{G^{--+}}}(u_{j})^{2}.
\end{eqnarray*}
Since $e_{_{G^{--+}}}(u) \leq 3$, we deduce that
\begin{eqnarray*}
  M^{1}_{ECI}(G^{--+}) &\leq& \sum\limits_{u_{i} \in V(G)} \big[3(n+1-deg_{_{G}}(u_{i}))^{2} \big] \\
  &+& \sum\limits_{u_{j}, u_{k} \in E(G)}\big[ 3 (m+3-(deg_{_{G}}(u_{j}) +deg_{_{G}}(u_{k})))^{2} \big] \\
  &=& 3n(n+1)^{2} + 3M_{1}(G) -6(n+1)(2m) + 3m(m+3)^{2} + 3F(G) \\
  &+& 6M_{2}(G) - 6(m+3)M_{1}(G)\\
  &=& 3F(G) + 6M_{2}(G) - (6m + 15)M_{1}(G) + 3n(n+1)^{2}\\
  &-&12m(n+1)+ 3m(m+3)^{2}.
\end{eqnarray*}

  \item[Case 3.]
  \begin{eqnarray*}
  ECI^{1}(G) &=& \sum\limits_{i=1}^{n} (e(v_{i})^{2}deg(v_{i}))\\
  &=& \sum\limits_{u_{i} \in V(G^{--+}) \cap V(G)} e_{_{G^{--+}}}(u_{i})^{2}\cdot deg_{_{G^{--+}}}(u_{i})\\
   &+& \sum\limits_{u_{j} \in V(G^{--+}) \cap E(G)} e_{_{G^{--+}}}(u_{j})^{2}\cdot deg_{_{G^{--+}}}(u_{j}).
\end{eqnarray*}
Since $e_{_{G^{--+}}}(u) \leq 3$, we obtain
\begin{eqnarray*}
  ECI^{1}(G^{--+}) &\leq& \sum\limits_{u_{i} \in V(G)} \big[3^{2}(n+1-deg_{G}(u_{i})) \big] \\
  &+& \sum\limits_{u_{j}, u_{k} \in E(G)}\big[ 3^{2} (m+3-(deg_{_{G}}(u_{j}) +deg_{_{G}}(u_{k}))) \big] \\
  &=& 9n(n+1) +9m(m+3)-18m -9M_{1}(G). 
\end{eqnarray*}

  \item[Case 4.] Finally we get
  \begin{eqnarray*}
  M^{1}_{ECI^{1}}(G) &=& \sum\limits_{i=1}^{n} (e(v_{i})^{2}deg(v_{i})^{2})\\
  &=& \sum\limits_{u_{i} \in V(G^{--+}) \cap V(G)} e_{_{G^{--+}}}(u_{i})^{2}\cdot deg_{_{G^{--+}}}(u_{i})^{2}\\
   &+& \sum\limits_{u_{j} \in V(G^{--+}) \cap E(G)} e_{_{G^{--+}}}(u_{j})^{2}\cdot deg_{_{G^{--+}}}(u_{j})^{2}.
\end{eqnarray*}
As $e_{_{G^{--+}}}(u) \leq 3$, we finally obtain that
\begin{eqnarray*}
  ECI^{1}(G^{--+}) &\leq& \sum\limits_{u_{i} \in V(G)} \big[3^{2}(n+1-deg_{_{G}}(u_{i}))^{2} \big] \\
  &+& \sum\limits_{u_{j}, u_{k} \in E(G)}\big[ 3^{2}(m+3-(deg_{_{G}}(u_{j}) +deg_{_{G}}(u_{k})))^{2} \big] \\
  &=& 9F(G) + 18M_{2}(G) - (18m+45)M_{1}(G) + 9n(n+1)^{2} \\
  &+&9m(m+3)^{2} -36m(n+1)
\end{eqnarray*}
\end{description}
as asserted.
\end{proof}

\begin{thm}
Let $G$ be an $(n,m)$ graph. Then
\begin{eqnarray*}
  I_{ECI}(G^{+-+}) &\leq& \frac{1}{12m}+ \frac{1}{3m(m+3)} -\frac{1}{3}M_{1}^{-1}(G),\\
  M_{ECI}^{1}(G^{+-+}) &\leq & 3F(G) + 6M_{2}(G) -6(m-3)M_{1}(G) + 3m(m+3)^{2},\\
  ECI^{1}(G^{+-+})&\leq& 36m + 9m(m+3) - 9M_{1}(G),\\
  M_{ECI^{1}}^{1}(G^{+-+})&\leq& 9F(G)+18M_{2}(G) -18(m-1)M_{1}(G) +9m(m+3)^{2}.
\end{eqnarray*}
\end{thm}
\begin{proof}
Let $G = (V,E)$ be a graph with $V(G) = \{v_{1}, v_{2}, v_{3}, \cdots, v_{n}\}$ and $E(G) = \{e_{1}, e_{2}, e_{3}, \cdots, e_{m}\}$. Then
$V(G^{+-+}) = \{v_{1}, v_{2}, v_{3}, \cdots, v_{n}, e_{1}, e_{2}, e_{3}, \cdots, e_{m}\}$. Clearly $|V(G^{+-+})|=m+n$ and $diam(G^{+-+}) \leq 3$. As $e_{G}(v) \leq diam(G)$ for every $v \in V(G)$, we get $e_{_{G^{+-+}}}(u) \leq 3$ for every $u \in G^{+-+}$. Let $u_{i}\in V(G^{+-+})$ be the corresponding vertex to $v_{i} \in V(G)$ and $u_{j} \in V(G^{+-+})$ be the corresponding vertex to $e_{j} \in E(G)$ in $G^{+-+}$. Then $deg_{_{G^{+-+}}}(u_{i}) = 2deg_{_{G}}(v_{i})$ and $deg_{_{G^{+-+}}}(u_{j}) =m+3-(deg_{_{G}}(v_{i})+ deg_{_{G}}(v_{j}))$ where $e_{j} = v_{i}v_{j}$. Therefore
\begin{description}
  \item[Case 1.]
  \begin{eqnarray*}
  I_{ECI}(G^{+-+}) &=& \sum\limits_{i=1}^{n} \frac{1}{e_{_{G^{+-+}}}(u)\cdot deg_{_{G^{+-+}}}(u)}\\
  &=& \sum\limits_{u_{i} \in V(G^{+-+}) \cap V(G)} \frac{1}{e_{_{G^{+-+}}}(u_{i})\cdot deg_{_{G^{+-+}}}(u_{i})} \\
  &+& \sum\limits_{u_{j} \in V(G^{+-+}) \cap E(G)} \frac{1}{e_{_{G^{+-+}}}(u_{j})\cdot deg_{_{G^{+-+}}}(u_{j})}
\end{eqnarray*}
As $e_{_{G^{+-+}}}(u) \leq 3$, we obtain
\begin{eqnarray*}
  I_{ECI}(G^{+-+}) &\leq&\sum\limits_{u_{i} \in V(G)} \frac{1}{\big[3\cdot (2(deg_{_{G}}(u_{i}))) \big]}\\
   &+& \sum\limits_{u_{j}, u_{k} \in E(G)}\frac{1}{\big[ 3\cdot (m+3-(deg_{_{G}}(u_{j}) +deg_{_{G}}(u_{k}))) \big]} \\
   &\leq& \frac{1}{12m}+ \frac{1}{3m(m+3)} -\frac{1}{3}M_{1}^{-1}(G).
\end{eqnarray*}

  \item[Case 2.]
  \begin{eqnarray*}
  M_{ECI}^{1}(G) &=& \sum\limits_{i=1}^{n} (e(v_{i})deg(v_{i})^{2})\\
  &=& \sum\limits_{u_{i} \in V(G^{+-+}) \cap V(G)} e_{_{G^{+-+}}}(u_{i})\cdot deg_{_{G^{+-+}}}(u_{i})^{2}\\
   &+& \sum\limits_{u_{j} \in V(G^{+-+}) \cap E(G)} e_{_{G^{+-+}}}(u_{j})\cdot deg_{_{G^{+-+}}}(u_{j})^{2}.
\end{eqnarray*}
As $e_{_{G^{+-+}}}(u) \leq 3$, we obtain
\begin{eqnarray*}
  M^{1}_{ECI}(G^{+-+}) &\leq& \sum\limits_{u_{i} \in V(G)} \big[3(2deg_{_{G}}(u_{i}))^{2} \big] \\
  &+& \sum\limits_{u_{j}, u_{k} \in E(G)}\big[ 3(m+3-(deg_{_{G}}(u_{j}) +deg_{_{G}}(u_{k})))^{2} \big] \\
  &=& 12M_{1}(G) +3m(m+3)^{2} + 3F(G) +6M_{2}(G) - 6(m+3)M_{1}(G) \\
  &=& 3F(G) + 6M_{2}(G) -6(m-3)M_{1}(G) + 3m(m+3)^{2}.
\end{eqnarray*}

  \item[Case 3.]
  \begin{eqnarray*}
  ECI^{1}(G) &=& \sum\limits_{i=1}^{n} (e(v_{i})^{2}deg(v_{i}))\\
  &=& \sum\limits_{u_{i} \in V(G^{+-+}) \cap V(G)} e_{_{G^{+-+}}}(u_{i})^{2}\cdot deg_{_{G^{+-+}}}(u_{i})\\
   &+& \sum\limits_{u_{j} \in V(G^{+-+}) \cap E(G)} e_{_{G^{+-+}}}(u_{j})^{2}\cdot deg_{_{G^{+-+}}}(u_{j}).
\end{eqnarray*}
Since $e_{_{G^{+-+}}}(u) \leq 3$, we deduce that
\begin{eqnarray*}
  ECI^{1}(G^{+-+}) &\leq& \sum\limits_{u_{i} \in V(G)} \big[3^{2}(2deg_{_{G}}(u_{i})) \big] \\
  &+& \sum\limits_{u_{j}, u_{k} \in E(G)}\big[ 3^{2}(m+3-(deg_{_{G}}(u_{j}) +deg_{_{G}}(u_{k}))) \big] \\
&=& 36m + 9m(m+3) - 9M_{1}(G).
\end{eqnarray*}

  \item[Case 4.] Finally, we have
  \begin{eqnarray*}
  M^{1}_{ECI^{1}}(G) &=& \sum\limits_{i=1}^{n} (e(v_{i})^{2}deg(v_{i})^{2})\\
  &=& \sum\limits_{u_{i} \in V(G^{+-+}) \cap V(G)} e_{_{G^{+-+}}}(u_{i})^{2}\cdot deg_{_{G^{+-+}}}(u_{i})^{2}\\
   &+& \sum\limits_{u_{j} \in V(G^{+-+}) \cap E(G)} e_{_{G^{+-+}}}(u_{j})^{2}\cdot deg_{_{G^{+-+}}}(u_{j})^{2}.
\end{eqnarray*}
As $e_{_{G^{+-+}}}(u) \leq 3$, we deduce that
\begin{eqnarray*}
  ECI^{1}(G^{+-+}) &\leq& \sum\limits_{u_{i} \in V(G)} \big[3^{2}(2deg_{_{G}}(u_{i}))^{2} \big] \\
  &+& \sum\limits_{u_{j}, u_{k} \in E(G)}\big[ 3^{2}(m+3-(deg_{_{G}}(u_{j}) +deg_{_{G}}(u_{k})))^{2} \big] \\
  &=& 9F(G)+18M_{2}(G) -18(m-1)M_{1}(G) +9m(m+3)^{2}.
\end{eqnarray*}
\end{description}
\end{proof}

\begin{thm}
Let $G$ be an $(n,m)$ graph. Then
\begin{eqnarray*}
  I_{ECI}(G^{-+-}) &\leq& \frac{1}{3n(m+n-1)}+\frac{1}{3m(n-4)}-\frac{1}{12m}+\frac{1}{3}M_{1}^{-1}(G),\\
  M_{ECI}^{1}(G^{-+-}) &\leq & 3F(G) + 6M_{2}(G) + 6(n+2)M_{1}(G) + 3n(m+n-1)^{2}\\
  &+&3m(n-4)^{2} -24(m+n-1),\\
  ECI^{1}(G^{-+-})&\leq& 9M_{1}(G) +9n(m+n) +9m(n-4) -36m,\\
  M_{ECI^{1}}^{1}(G^{-+-})&\leq& 9F(G) + 18M_{2}(G) + 18(n-3)M_{1}(G) + 9n(m+n-1)^{2} \\
  &+& 9m(n-4)^{2}- 36m(m+n-1).
\end{eqnarray*}
\end{thm}
\begin{proof}
Let $G = (V,E)$ with $V(G) = \{v_{1}, v_{2}, v_{3}, \cdots, v_{n}\}$ and $E(G) = \{e_{1}, e_{2}, e_{3}, \cdots, e_{m}\}$. Then
$V(G^{-+-}) = \{v_{1}, v_{2}, v_{3}, \cdots, v_{n}, e_{1}, e_{2}, e_{3}, \cdots, e_{m}\}$. As $|V(G^{-+-})|=m+n$ and $diam(G^{-+-}) \leq 3$, since $e_{G}(v) \leq diam(G)$ for every $v \in V(G)$, we obtain $e_{_{G^{-+-}}}(u) \leq 3$ for every $u \in G^{-+-}$. Let $u_{i}\in V(G^{-+-})$ be the corresponding vertex to $v_{i} \in V(G)$ and $u_{j} \in V(G^{-+-})$ be the corresponding vertex to $e_{j} \in E(G)$ in $G^{-+-}$. Then $deg_{_{G^{-+-}}}(u_{i}) =m+n-1- 2deg_{_{G}}(v_{i})$ and $deg_{_{G^{-+-}}}(u_{j}) =n-4+(deg_{_{G}}(v_{i})+ deg_{_{G}}(v_{j}))$ where $e_{j} = v_{i}v_{j}$. Therefore we have the following cases to consider:
\begin{description}
  \item[Case 1.]
  \begin{eqnarray*}
  I_{ECI}(G^{-+-}) &=& \sum\limits_{i=1}^{n} \frac{1}{e_{_{G^{-+-}}}(u)\cdot deg_{_{G^{-+-}}}(u)}\\
  &=& \sum\limits_{u_{i} \in V(G^{-+-}) \cap V(G)} \frac{1}{e_{_{G^{-+-}}}(u_{i})\cdot deg_{_{G^{-+-}}}(u_{i})} \\
  &+& \sum\limits_{u_{j} \in V(G^{-+-}) \cap E(G)} \frac{1}{e_{_{G^{-+-}}}(u_{j})\cdot deg_{_{G^{-+-}}}(u_{j})}.
\end{eqnarray*}
Since $e_{_{G^{-+-}}}(u) \leq 3$, we have
\begin{eqnarray*}
  I_{ECI}(G^{-+-}) &\leq&\sum\limits_{u_{i} \in V(G)} \frac{1}{\big[3\cdot (m+n-1-2(deg_{_{G}}(u_{i}))) \big]}\\
   &+& \sum\limits_{u_{j}, u_{k} \in E(G)}\frac{1}{\big[ 3\cdot (n-4+(deg_{_{G}}(u_{j}) +deg_{_{G}}(u_{k}))) \big]} \\
   &\leq& \frac{1}{3n(m+n-1)}+\frac{1}{3m(n-4)}-\frac{1}{12m}+\frac{1}{3}M_{1}^{-1}(G).
\end{eqnarray*}

  \item[Case 2.]
  \begin{eqnarray*}
  M_{ECI}^{1}(G) &=& \sum\limits_{i=1}^{n} e(v_{i})deg(v_{i})^{2}\\
  &=& \sum\limits_{u_{i} \in V(G^{-+-}) \cap V(G)} e_{_{G^{-+-}}}(u_{i})\cdot deg_{_{G^{-+-}}}(u_{i})^{2}\\
   &+& \sum\limits_{u_{j} \in V(G^{-+-}) \cap E(G)} e_{_{G^{-+-}}}(u_{j})\cdot deg_{_{G^{-+-}}}(u_{j})^{2}.
\end{eqnarray*}
Since $e_{_{G^{-+-}}}(u) \leq 3$, we have
\begin{eqnarray*}
  M^{1}_{ECI}(G^{-+-}) &\leq& \sum\limits_{u_{i} \in V(G)} \big[3(m+n-1-2deg_{_{G}}(u_{i}))^{2} \big] \\
  &+& \sum\limits_{u_{j}, u_{k} \in E(G)}\big[ 3(n-4 + (deg_{_{G}}(u_{j}) +deg_{_{G}}(u_{k})))^{2} \big] \\
  &=& 3n(m+n-1)^{2} + 12M_{1}(G) -24(m+n-1)m + 3m(n-4)^{2} + 3F(G) \\
  &+& 6M_{2}(G) + 6(n-4)M_{1}(G)\\
  &=& 3F(G) + 6M_{2}(G) + 6(n+2)M_{1}(G) + 3n(m+n-1)^{2} \\
  &+&3m(n-4)^{2} -24(m+n-1).
\end{eqnarray*}

  \item[Case 3.]
  \begin{eqnarray*}
  ECI^{1}(G) &=& \sum\limits_{i=1}^{n} e(v_{i})^{2}deg(v_{i})\\
  &=& \sum\limits_{u_{i} \in V(G^{-+-}) \cap V(G)} e_{_{G^{-+-}}}(u_{i})^{2}\cdot deg_{_{G^{-+-}}}(u_{i})\\
   &+& \sum\limits_{u_{j} \in V(G^{-+-}) \cap E(G)} e_{_{G^{-+-}}}(u_{j})^{2}\cdot deg_{_{G^{-+-}}}(u_{j}).
\end{eqnarray*}
As $e_{_{G^{-+-}}}(u) \leq 3$, we obtain
\begin{eqnarray*}
  ECI^{1}(G^{-+-}) &\leq& \sum\limits_{u_{i} \in V(G)} \big[3^{2}(m+n-2deg_{_{G}}(u_{i})) \big] \\
  &+& \sum\limits_{u_{j}, u_{k} \in E(G)}\big[ 3^{2}(n-4 +(deg_{_{G}}(u_{j}) +deg_{_{G}}(u_{k}))) \big] \\
 &=& 9M_{1}(G) +9n(m+n) +9m(n-4) -36m.
\end{eqnarray*}

  \item[Case 4.]
  \begin{eqnarray*}
  M^{1}_{ECI^{1}}(G) &=& \sum\limits_{i=1}^{n} (e(v_{i})^{2}deg(v_{i})^{2})\\
  &=& \sum\limits_{u_{i} \in V(G^{-+-}) \cap V(G)} e_{_{G^{-+-}}}(u_{i})^{2}\cdot deg_{_{G^{-+-}}}(u_{i})^{2}\\
   &+& \sum\limits_{u_{j} \in V(G^{-+-}) \cap E(G)} e_{_{G^{-+-}}}(u_{j})^{2}\cdot deg_{_{G^{-+-}}}(u_{j})^{2}.
\end{eqnarray*}
Since $e_{_{G^{-+-}}}(u) \leq 3$, we have
\begin{eqnarray*}
  ECI^{1}(G^{-+-}) &\leq& \sum\limits_{u_{i} \in V(G)} \big[3^{2}(m+n-1-2deg_{_{G}}(u_{i}))^{2} \big] \\
  &+& \sum\limits_{u_{j}, u_{k} \in E(G)}\big[ 3^{2}(n-4+(deg_{_{G}}(u_{j}) +deg_{_{G}}(u_{k})))^{2} \big] \\
  &=& 9F(G) + 18M_{2}(G) + 18(n-3)M_{1}(G) + 9n(m+n-1)^{2} \\
  &+& 9m(n-4)^{2}- 36m(m+n-1)
\end{eqnarray*}
\end{description}
as asserted.
\end{proof}

\begin{thm}
Let $G$ be an $(n,m)$ graph. Then
\begin{eqnarray*}
  I_{ECI}(G^{-++}) &\leq& \frac{1}{3n(n-1)}+\frac{1}{3}M_{1}^{-1}(G),\\
  M_{ECI}^{1}(G^{-++}) &\leq & 3F(G) + 6M_{2}(G) + 3n(n-1)^{2},\\
  ECI^{1}(G^{-++})&\leq& 9M_{1}(G) + 9n(n-1),\\
  M_{ECI^{1}}^{1}(G^{-++})&\leq& 9F(G) + 18M_{2}(G) + 9n(n-1)^{2}.
\end{eqnarray*}
\end{thm}
\begin{proof}
Let $G = (V,E)$ with $V(G) = \{v_{1}, v_{2}, v_{3}, \cdots, v_{n}\}$ and $E(G) = \{e_{1}, e_{2}, e_{3}, \cdots, e_{m}\}$. Then
$V(G^{-++}) = \{v_{1}, v_{2}, v_{3}, \cdots, v_{n}, e_{1}, e_{2}, e_{3}, \cdots, e_{m}\}$. Clearly $|V(G^{-++})|=m+n$ and $diam(G^{-++}) \leq 3$. Since $e_{G}(v) \leq diam(G)$ for every $v \in V(G)$, we have $e_{_{G^{-++}}}(u) \leq 3$ for every $u \in G^{-++}$. Let $u_{i}\in V(G^{-++})$ be the corresponding vertex to $v_{i} \in V(G)$ and $u_{j} \in V(G^{-++})$ be the corresponding vertex to $e_{j} \in E(G)$ in $G^{-++}$. Then $deg_{_{G^{-++}}}(u_{i}) =n-1$ and $deg_{_{G^{-++}}}(u_{j}) =(deg_{_{G}}(v_{i})+ deg_{_{G}}(v_{j}))$ where $e_{j} = v_{i}v_{j}$. Therefore
\begin{description}
  \item[Case 1.]
 \begin{eqnarray*}
  I_{ECI}(G^{-++}) &=& \sum\limits_{i=1}^{n} \frac{1}{e_{_{G^{-++}}}(u)\cdot deg_{_{G^{-++}}}(u)}\\
  &=& \sum\limits_{u_{i} \in V(G^{-++}) \cap V(G)} \frac{1}{e_{_{G^{-++}}}(u_{i})\cdot deg_{_{G^{-++}}}(u_{i})} \\
  &+& \sum\limits_{u_{j} \in V(G^{-++}) \cap E(G)} \frac{1}{e_{_{G^{-++}}}(u_{j})\cdot deg_{_{G^{-++}}}(u_{j})}.
\end{eqnarray*}
Since $e_{_{G^{-++}}}(u) \leq 3$, we obtain
\begin{eqnarray*}
  I_{ECI}(G^{-++}) &\leq&\sum\limits_{u_{i} \in V(G)} \frac{1}{\big[3\cdot (n-1) \big]}\\
   &+& \sum\limits_{u_{j}, u_{k} \in E(G)}\frac{1}{\big[ 3\cdot (deg_{_{G}}(u_{j}) +deg_{_{G}}(u_{k})) \big]} \\
   &\leq& \frac{1}{3n(n-1)}+\frac{1}{3}M_{1}^{-1}(G).
\end{eqnarray*}

  \item[Case 2.]
  \begin{eqnarray*}
  M_{ECI}^{1}(G) &=& \sum\limits_{i=1}^{n} (e(v_{i})deg(v_{i})^{2})\\
  &=& \sum\limits_{u_{i} \in V(G^{-++}) \cap V(G)} e_{_{G^{-++}}}(u_{i})\cdot deg_{_{G^{-++}}}(u_{i})^{2}\\
   &+& \sum\limits_{u_{j} \in V(G^{-++}) \cap E(G)} e_{_{G^{-++}}}(u_{j})\cdot deg_{_{G^{-++}}}(u_{j})^{2}.
\end{eqnarray*}
As $e_{_{G^{-++}}}(u) \leq 3$, we obtain
\begin{eqnarray*}
  M^{1}_{ECI}(G^{-++}) &\leq& \sum\limits_{u_{i} \in V(G)} \big[3 (n-1)^{2} \big] \\
  &+& \sum\limits_{u_{j}, u_{k} \in E(G)}\big[ 3(deg_{_{G}}(u_{j}) +deg_{_{G}}(u_{k}))^{2} \big] \\
  &=& 3F(G) + 6M_{2}(G) + 3n(n-1)^{2}.
\end{eqnarray*}

  \item[Case 3.]
  \begin{eqnarray*}
  ECI^{1}(G) &=& \sum\limits_{i=1}^{n} (e(v_{i})^{2}deg(v_{i}))\\
  &=& \sum\limits_{u_{i} \in V(G^{-++}) \cap V(G)} e_{_{G^{-++}}}(u_{i})^{2}\cdot deg_{_{G^{-++}}}(u_{i})\\
   &+& \sum\limits_{u_{j} \in V(G^{-++}) \cap E(G)} e_{_{G^{-++}}}(u_{j})^{2}\cdot deg_{_{G^{-++}}}(u_{j}).
\end{eqnarray*}
The fact that $e_{_{G^{-++}}}(u) \leq 3$ implies that 
\begin{eqnarray*}
  ECI^{1}(G^{-++}) &\leq& \sum\limits_{u_{i} \in V(G)} \big[3^{2}(n-1) \big] \\
  &+& \sum\limits_{u_{j}, u_{k} \in E(G)}\big[ 3^{2}(deg_{_{G}}(u_{j}) +deg_{_{G}}(u_{k})) \big] \\
&=& 9M_{1}(G) + 9n(n-1).
\end{eqnarray*}

  \item[Case 4.] Finally we have
  \begin{eqnarray*}
  M^{1}_{ECI^{1}}(G) &=& \sum\limits_{i=1}^{n} (e(v_{i})^{2}deg(v_{i})^{2})\\
  &=& \sum\limits_{u_{i} \in V(G^{-++}) \cap V(G)} e_{_{G^{-++}}}(u_{i})^{2}\cdot deg_{_{G^{-++}}}(u_{i})^{2}\\
   &+& \sum\limits_{u_{j} \in V(G^{-++}) \cap E(G)} e_{_{G^{-++}}}(u_{j})^{2}\cdot deg_{_{G^{-++}}}(u_{j})^{2}.
\end{eqnarray*}
Since $e_{_{G^{-++}}}(u) \leq 3$, we get
\begin{eqnarray*}
  ECI^{1}(G^{-++}) &\leq& \sum\limits_{u_{i} \in V(G)} \big[3^{2}(n-1)^{2} \big] \\
  &+& \sum\limits_{u_{j}, u_{k} \in E(G)}\big[3^{2}(deg_{_{G}}(u_{j}) +deg_{_{G}}(u_{k}))^{2} \big] \\
  &=& 9F(G) + 18M_{2}(G) + 9n(n-1)^{2}
\end{eqnarray*}
\end{description}
as asserted.
\end{proof}

\begin{thm}
Let $G$ be an $(n,m)$ graph. Then
\begin{eqnarray*}
  I_{ECI}(G^{+--}) &\leq& \frac{1}{4mn}+\frac{1}{4m(m+n-1)}-\frac{1}{4}M_{1}^{-1}(G),\\
  M_{ECI}^{1}(G^{+--}) &\leq &  4F(G) + 8M_{2}(G) - 8(m+n-1)M_{1}(G) + 4nm^{2} + 4m(m+n-1)^{2},\\
  ECI^{1}(G^{+--})&\leq& 16mn +16m(m+n-1) -16M_{1}(G),\\
  M_{ECI^{1}}^{1}(G^{+--})&\leq& 16F(G) +32M_{2}(G) -32(m+n-1)M_{1}(G) + 16m^{2}n + 16m(m+n-1)^{2}.
\end{eqnarray*}
\end{thm}
\begin{proof}
Let $G = (V,E)$ with $V(G) = \{v_{1}, v_{2}, v_{3}, \cdots, v_{n}\}$ and $E(G) = \{e_{1}, e_{2}, e_{3}, \cdots, e_{m}\}$. Then
$V(G^{+--}) = \{v_{1}, v_{2}, v_{3}, \cdots, v_{n}, e_{1}, e_{2}, e_{3}, \cdots, e_{m}\}$. Clearly $|V(G^{+--})|=m+n$ and $diam(G^{+--}) \leq 4$. Since $e_{G}(v) \leq diam(G)$ for every $v \in V(G)$, we obtain $e_{_{G^{+--}}}(u) \leq 4$ for every $u \in G^{+--}$. Let $u_{i}\in V(G^{+--})$ be the corresponding vertex to $v_{i} \in V(G)$ and $u_{j} \in V(G^{+--})$ be the corresponding vertex to $e_{j} \in E(G)$ in $G^{+--}$. Then $deg_{_{G^{+--}}}(u_{i}) =m$ and $deg_{_{G^{+--}}}(u_{j}) =m+n-1-(deg_{_{G}}(v_{i})+ deg_{_{G}}(v_{j}))$ where $e_{j} = v_{i}v_{j}$. Therefore
\begin{description}
  \item[Case 1.]
\begin{eqnarray*}
  I_{ECI}(G^{+--}) &=& \sum\limits_{i=1}^{n} \frac{1}{e_{_{G^{+--}}}(u)\cdot deg_{_{G^{+--}}}(u)}\\
  &=& \sum\limits_{u_{i} \in V(G^{+--}) \cap V(G)} \frac{1}{e_{_{G^{+--}}}(u_{i})\cdot deg_{_{G^{+--}}}(u_{i})}\\
   &+& \sum\limits_{u_{j} \in V(G^{+--}) \cap E(G)} \frac{1}{e_{_{G^{+--}}}(u_{j})\cdot deg_{_{G^{+--}}}(u_{j})}.
\end{eqnarray*}
Since $e_{_{G^{+--}}}(u) \leq 4$, we have
\begin{eqnarray*}
  I_{ECI}(G^{+--}) &\leq&\sum\limits_{u_{i} \in V(G)} \frac{1}{\big[4\cdot (m) \big]}\\
   &+& \sum\limits_{u_{j}, u_{k} \in E(G)}\frac{1}{\big[ 4\cdot(m+n-1- (deg_{_{G}}(u_{j}) +deg_{_{G}}(u_{k}))) \big]} \\
   &\leq& \frac{1}{4mn}+\frac{1}{4m(m+n-1)}-\frac{1}{4}M_{1}^{-1}(G).
\end{eqnarray*}

  \item[Case 2.]
  \begin{eqnarray*}
  M_{ECI}^{1}(G) &=& \sum\limits_{i=1}^{n} (e(v_{i})deg(v_{i})^{2})\\
  &=& \sum\limits_{u_{i} \in V(G^{+--}) \cap V(G)} e_{_{G^{+--}}}(u_{i})\cdot deg_{_{G^{+--}}}(u_{i})^{2}\\
   &+& \sum\limits_{u_{j} \in V(G^{+--}) \cap E(G)} e_{_{G^{+--}}}(u_{j})\cdot deg_{_{G^{+--}}}(u_{j})^{2}.
\end{eqnarray*}
Since $e_{_{G^{+--}}}(u) \leq 4$, we obtain
\begin{eqnarray*}
  M^{1}_{ECI}(G^{+--}) &\leq& \sum\limits_{u_{i} \in V(G)} \big[4m^{2} \big] \\
  &+& \sum\limits_{u_{j}, u_{k} \in E(G)}\big[ 4 (m+n-1 -(deg_{_{G}}(u_{j}) +deg_{_{G}}(u_{k})))^{2} \big] \\
  &=& 4nm^{2} + 4m(m+n-1)^{2} + 4F(G) + 8M_{2}(G) -8(m+n-1)M_{1}(G) \\
  &=& 4F(G) + 8M_{2}(G) - 8(m+n-1)M_{1}(G) + 4nm^{2} + 4m(m+n-1)^{2}.
\end{eqnarray*}

  \item[Case 3.]
  \begin{eqnarray*}
  ECI^{1}(G) &=& \sum\limits_{i=1}^{n} (e(v_{i})^{2}deg(v_{i}))\\
  &=& \sum\limits_{u_{i} \in V(G^{+--}) \cap V(G)} e_{_{G^{+--}}}(u_{i})^{2}\cdot deg_{_{G^{+--}}}(u_{i})\\
   &+& \sum\limits_{u_{j} \in V(G^{+--}) \cap E(G)} e_{_{G^{+--}}}(u_{j})^{2}\cdot deg_{_{G^{+--}}}(u_{j}).
\end{eqnarray*}
As $e_{_{G^{+--}}}(u) \leq 4$, we obtain the following:
\begin{eqnarray*}
  ECI^{1}(G^{+--}) &\leq& \sum\limits_{u_{i} \in V(G)} \big[4^{2}(m) \big] \\
  &+& \sum\limits_{u_{j}, u_{k} \in E(G)}\big[ 4^{2}(m+n-1-(deg_{_{G}}(u_{j}) +deg_{_{G}}(u_{k}))) \big] \\
&=& 16mn +16m(m+n-1) -16M_{1}(G).
\end{eqnarray*}

  \item[Case 4.] Finally we have
  \begin{eqnarray*}
  M^{1}_{ECI^{1}}(G) &=& \sum\limits_{i=1}^{n} (e(v_{i})^{2}deg(v_{i})^{2})\\
  &=& \sum\limits_{u_{i} \in V(G^{+--}) \cap V(G)} e_{_{G^{+--}}}(u_{i})^{2}\cdot deg_{_{G^{+--}}}(u_{i})^{2}\\
   &+& \sum\limits_{u_{j} \in V(G^{+--}) \cap E(G)} e_{_{G^{+--}}}(u_{j})^{2}\cdot deg_{_{G^{+--}}}(u_{j})^{2}.
\end{eqnarray*}
Since $e_{_{G^{+--}}}(u) \leq 4$, we have
\begin{eqnarray*}
  ECI^{1}(G^{+--}) &\leq& \sum\limits_{u_{i} \in V(G)} \big[4^{2}\cdot (m)^{2} \big] \\
  &+& \sum\limits_{u_{j}, u_{k} \in E(G)}\big[ 4^{2}(m+n-1-(deg_{_{G}}(u_{j}) +deg_{_{G}}(u_{k})))^{2} \big] \\
  &=& 16F(G) +32M_{2}(G) -32(m+n-1)M_{1}(G) + 16m^{2}n + 16m(m+n-1)^{2}.
\end{eqnarray*}
\end{description}
\end{proof}

\noindent \textbf{Conclusion:} In this paper, we have obtained bounds for four newly introduced eccentric-based topological indices for eight kinds of transformation graphs of total graphs.

\end{document}